\def\version{{\tiny Version \today (typeset: \today)}}
\def\version{}

\documentclass[reqno,twoside,11pt]{amsart}
\usepackage{cite}
\usepackage{amsmath}
\usepackage{amsfonts}
\usepackage{amssymb}
\usepackage{epsfig}
\usepackage{verbatim}
\usepackage{color}
\usepackage{bbm}
\IfFileExists{epsf.def}{\input epsf.def}{\usepackage{epsf}}

\IfFileExists{mathptmx.sty}{\usepackage{mathptmx}}{\usepackage{mathpazo}}
\usepackage{mathrsfs}

\DeclareFontFamily{OT1}{eusb}{} \DeclareFontShape{OT1}{eusb}{m}{n} {<5> <6> <7> <8> <9> <10> <11> <12> <14.4> eusb10}{}
\DeclareMathAlphabet{\eusb}{OT1}{eusb}{m}{n}

\DeclareFontFamily{OT1}{eusm}{} \DeclareFontShape{OT1}{eusm}{m}{n} {<5> <6> <7> <8> <9> <10> <11> <12> <14.4> eusm10}{}
\DeclareMathAlphabet{\eusm}{OT1}{eusm}{m}{n}

\DeclareFontFamily{OT1}{eufm}{} \DeclareFontShape{OT1}{eufm}{m}{n} {<5> <6> <7> <8> <9> <10> <11> <12> <14.4> eufm10}{}
\DeclareMathAlphabet{\mathfrak}{OT1}{eufm}{m}{n}

\DeclareFontFamily{OT1}{fraktura}{}
\DeclareFontShape{OT1}{fraktura}{m}{n} {<5> <6> <7> <8> <9> <10> <11> <12> <13> <14.4> [1.1] eufm10}{}
\DeclareMathAlphabet{\fraktura}{OT1}{fraktura}{m}{n}

\DeclareFontFamily{OT1}{cmfi}{} \DeclareFontShape{OT1}{cmfi}{m}{n} {<5> <6> <7> <8> <9> <10> <11> <12> <13> <14.4> [0.9] cmfi10}{}
\DeclareMathAlphabet{\cmfi}{OT1}{cmfi}{b}{n}

\DeclareFontFamily{OT1}{cmss}{} \DeclareFontShape{OT1}{cmss}{m}{n} {<5> <6> <7> <8> <9> <10> <11> <12> <13> <14.4> cmss10}{}
\DeclareMathAlphabet{\cmss}{OT1}{cmss}{m}{n}

\newtheoremstyle{thm}{1.5ex}{1.5ex}{\itshape\rmfamily}{} {\bfseries\rmfamily}{}{2ex}{}

\newtheoremstyle{def}{1.5ex}{1.5ex}{\rmfamily\sl}{} {\bfseries\rmfamily}{}{2ex}{}

\newtheoremstyle{rem}{1.3ex}{1.3ex}{\rmfamily}{} {\bfseries\rmfamily}{}{2ex}{}

\newtheoremstyle{ass}{1.5ex}{1.5ex}{\rmfamily\sl}{} {\bfseries\rmfamily}{}{2ex}{}

\newenvironment{proofsect}[1] {\vskip0.1cm\noindent{\rmfamily\itshape#1.}}{\qed\vspace{0.15cm}}

\theoremstyle{thm}
\newtheorem{theorem}{Theorem}[section]
\newtheorem{lemma}[theorem]{Lemma}
\newtheorem{proposition}[theorem]{Proposition}

\newtheorem*{Main Theorem}{Main Theorem.}

\newtheoremstyle{named}{}{}{\itshape}{}{\bfseries}{}{.5em}{\thmnote{#3}}
\theoremstyle{named}

\theoremstyle{def}
\newtheorem{definition}[theorem]{Definition}

\theoremstyle{rem}
\newtheorem{remark}[theorem]{{Remark}}

\numberwithin{equation}{section}

\renewcommand{\section}{\secdef\sct\sect}
\newcommand{\sct}[2][default]{\refstepcounter{section}
\addcontentsline{toc}{section}
{{\tocsection {}{\thesection}{\!\!\!\!#1\dotfill}}{}}
\vspace{0.7cm}
\centerline{ 
\scshape\arabic{section}.\ #1} \nopagebreak \vspace{0.2cm}}
\newcommand{\sect}[1]{
\vspace{0.4cm} \centerline{\large\scshape\rmfamily #1}
\vspace{0.2cm}}

\renewcommand{\subsection}{\secdef\subsct\sbsect}
\newcommand{\subsct}[2][default]{\refstepcounter{subsection}
\addcontentsline{toc}{subsection}
{{\tocsection{\!\!}{\hspace{1.2em}\thesubsection}{\!\!\!\!#1\dotfill}}{}}
\nopagebreak\vspace{0.45\baselineskip} {\flushleft\bf
\arabic{section}.\arabic{subsection}~\bf #1.~}
\\*[3mm]\noindent
\nopagebreak}
\newcommand{\sbsect}[1]{
\vspace{0.1cm}\noindent
\textbf{#1.~}\vspace{0.1cm}}

\renewcommand{\subsubsection}{%
\secdef \subsubsect\sbsbsect}
\newcommand{\subsubsect}[2][default]{%
\refstepcounter{subsubsection} 
\addcontentsline{toc}{subsubsection}{{\tocsection{\!\!}
{\hspace{3.05em}\thesubsubsection}{\!\!\!\!#1\dotfill}}{}}
\nopagebreak
\vspace{0.15\baselineskip} \nopagebreak {\flushleft\rmfamily
\itshape\arabic{section}.\arabic{subsection}.\arabic{subsubsection}
\ \rmfamily #1\/.}\ }
\newcommand{\sbsbsect}[1]{\vspace{0.1cm}\noindent
\rmfamily \itshape
\arabic{section}.\arabic{subsection}.\arabic{subsubsection} \
\sffamily #1\/.\ }

\renewcommand{\caption}[1]{%
\vglue0.5cm
\refstepcounter{figure}
\begin{minipage}{0.9\textwidth}\small {\sc Figure~\thefigure. }#1\end{minipage}}

\newcommand{\dist}{\operatorname{dist}}
\newcommand{\supp}{\operatorname{supp}}

\newcommand{\textd}{\text{\rm d}\mkern0.5mu}

\newcommand{\mini}{{\text{\rm min}}}
\newcommand{\maxi}{{\text{\rm max}}}

\newcommand{\BH}{{\mathbb{H}}}

\newcommand{\BR}{{\mathbb{R}}}

\newcommand{\BZ}{{\mathbb{Z}}}

\newcommand{\CF}{{\mathcal{F}}}

\newcommand{\CM}{{\mathcal{M}}}

\newcommand{\CP}{{\mathcal{P}}}

\newcommand{\CU}{{\mathcal{U}}}

\newcommand{\bae}{\begin{equation}\begin{aligned}}
\newcommand{\eae}{\end{aligned}\end{equation}}

\DeclareFontFamily{OML}{rsfs}{\skewchar\font'177}
\DeclareFontShape{OML}{rsfs}{m}{n}{ <5> <6> rsfs5 <7> <8> <9>
rsfs7 <10> <10.95> <12> <14.4> <17.28> <20.74> <24.88> rsfs10 }{}
\DeclareMathAlphabet{\mathfs}{OML}{rsfs}{m}{n}

\newcommand{\cc}{{\text{\rm c}}}

\def\myffrac#1#2 in #3{\raise 2.6pt\hbox{$#3 #1$}\mkern-1.5mu\raise 0.8pt\hbox{$#3/$}\mkern-1.1mu\lower 1.5pt\hbox{$#3 #2$}}

\newcommand{\wh}{\widehat}

\newcommand{\distH}{{\text{\rm dist}}_{\text{\rm H}}}

\begin{document}

\title[Eigenvalues under perimeter constraint\hfill  \version \hfill]{\large 
Shapes of drums with lowest base frequency\\under non-isotropic perimeter constraints}

\author[\hfill  \version \hfill Biskup and Procaccia]
{Marek~Biskup$^1$ and Eviatar B.\ Procaccia$^2$}
\thanks{\hglue-4.5mm\fontsize{9.6}{9.6}\selectfont\copyright\,\textrm{2016}\ \textrm{M.~Biskup, E.B.~Procaccia.
Reproduction, by any means, of the entire
article for non-commercial purposes is permitted without charge.\vspace{2mm}}}
\maketitle

\vspace{-4mm}
\centerline{$^1$\textit{Department of Mathematics, UCLA, Los Angeles, California, USA}}
\centerline{$^2$\textit{Department of Mathematics, Texas A\&M University, College Station, Texas, USA}}

\begin{abstract}
We study the minimizers of the sum of the principal Dirichlet eigenvalue of the negative Laplacian and the perimeter with respect to a general norm in the class of Jordan domains in the plane. This is equivalent (modulo scaling) to minimizing the said eigenvalue (or the base frequency of a drum of this shape) subject to a hard constraint on the perimeter. We show that, for all norms, a minimizer exists, is unique up to spatial translations and is convex but not necessarily smooth. We give conditions on the norm that characterize the appearance of facets and corners. We also demonstrate that near minimizers have to be close to the optimal ones in the Hausdorff distance. Our motivation for considering this class of variational problems comes from a study of random walks in random environment interacting through the boundary of their~support. 
\end{abstract}
\maketitle

\vglue3mm

\section{Main results}
\noindent
The aim of this paper is to analyze the planar domains that minimize the principal harmonic frequency (i.e., the lowest Dirichlet eigenvalue of the negative Laplacian or the base frequency of the drum of this shape) subject to a hard constraint on the perimeter defined using a general norm. Let us introduce the problem and state our results on it first and leave the discussion of our motivation as well as connections to existing literature to a later section. 

\subsection{The model and main theorems}
We will consider the following class of admissible domains
\begin{equation}
\label{E:1.11}
\CU:=\bigl\{U\subset\BR^2\colon\text{Jordan domain}, 0\in U
\bigr\}.
\end{equation}
Note that each~$U\in\CU$ is open, bounded, connected and simply connected with~$\partial U$ being the trace of a simple closed curve.
The principal Dirichlet eigenvalue of the negative Laplacian is for any open $U\subset\BR^2$ given by\begin{equation}
\label{E:1.4}
\lambda(U):=\inf\Bigl\{\Vert\nabla g\Vert_2^2\colon g\in C^\infty(\BR^2),\,\supp(g)\subset U,
\Vert g\Vert_2=1\Bigr\},
\end{equation}
where $\Vert\cdot\Vert_2$ denotes the $L^2$ norm on~$\BR^2$ with respect to the Lebesgue measure.
Next fix a norm~$\rho$ on~$\BR^2$ and denote by $\gamma\colon[0,1]\to\BR^2$ the curve constituting $\partial U$ (where $\gamma(0)=\gamma(1)$). Then we define the $\rho$-perimeter~$\CP(U)$ of~$U\in\CU$ by
\begin{equation}
\label{E:1.12}
\mathfs{P}(U):=\,\sup_{n\ge1}\,\sup_{\begin{subarray}{c}
t_0,\dots,t_n\in[0,1]\\0= t_0<t_1<\ldots<t_n=1
\end{subarray}}\,
\,\sum_{i=1}^n\rho\bigl(\gamma(t_i)-\gamma(t_{i-1})\bigr).
\end{equation}
Our main point of interest are the minimizers of~$\lambda(U)$ subject to the constraint~$\CP(U)=a$.

Instead of dealing with the hard constraint, throughout this paper we will focus on minimizing the functional
\begin{equation}
\label{E:1.13}
\CF(U):=\lambda(U)+\mathfs{P}(U)
\end{equation}
over all sets~$U\in\CU$. Our main result concerning this problem is as follows:

\begin{theorem}[Existence and uniqueness of minimizers]
\label{thm-0}
For any choice of the norm~$\rho$, the functional $U\mapsto\CF(U)$ achieves its minimum on~$\CU$. Moreover, the minimizer is unique modulo translates: There is $U_0\in\CU$ such that
\begin{equation}
\label{E:1.14}
\CM:=\bigl\{U\in\CU\colon \CF(U)=\min\CF\bigr\}=\bigl\{x+U_0\colon -x\in U_0\bigr\}.
\end{equation}
The set $U_0$ is convex and has all symmetries of~$\BR^2$ that preserve~$\rho$.
\end{theorem}

\noindent
To see that this does resolve also the problem with the hard constraint on the perimeter we note that the sheer existence of a minimizer~$U_0$ of~$\CF$ implies
\begin{equation}
\inf\bigl\{\lambda(U)\colon\CP(U)=\CP(U_0)\bigr\}=\lambda(U_0).
\end{equation}
In light of \eqref{E:1.14} and the scaling relations
\begin{equation}
\label{E:scaling}
\lambda(tU)=t^{-2}\lambda(U)
\quad\text{and}\quad
\CP(tU)=t\CP(U),
\end{equation}
where $tU:=\{tx\colon x\in U\}$, every minimizer of~$\lambda(U)$ subject to~$\CP(U)=a$ is then a translate of~$tU_0$ where $t:=a/\CP(U_0)$.

A natural next problem is the determination of the shape~$U_0$ for specific choices of~$\rho$. Here are three examples where~$U_0$ can explicitly be found:
\begin{enumerate}
\item[(1)] $\rho$ := Euclidean norm $\Rightarrow$ $U_0$ = a Euclidean ball,
\item[(2)] $\rho$ := $\ell^1$-norm $\Rightarrow$ $U_0$ = an $\ell^\infty$-ball,
\item[(3)] $\rho$ := $\ell^\infty$-norm $\Rightarrow$ $U_0$ = an $\ell^1$-ball,
\end{enumerate}
see Remark~\ref{rem1} for a justification.

Recall that non-empty open sets $U,V\subset\BR^2$ are said to be \emph{homothetic} if we have $U=x+tV$ for some $x\in\BR^2$ and some $t>0$ and let $|U|$ denote the Lebesgue measure of~$U$. Comparing (1-3) above with the \emph{isoperimetric problem} for the perimeter defined using the norm~$\rho$,
\begin{equation}
\label{E:1.7w}
\inf\bigl\{\CP(U)\colon U\in\CU,\,|U|=1\bigr\},
\end{equation}
which is solved (see Taylor~\cite{Taylor1974,Taylor1975}) by translates of $t\widetilde U_0$ where
\begin{equation}
\label{E:wulff}
\widetilde U_0:=\bigcap_{y\in\BR^2}\bigl\{x\in\BR^2\colon x\cdot y<\rho(y)\bigr\}
\end{equation}
and $t$ is such that $|\widetilde U_0|=1$, it may come as a surprise that in each of the three cases above the minimizers of \eqref{E:1.13} are homothetic to the  minimizers of \eqref{E:1.7w}. However, the above examples are very special and the connection fails in general:

\begin{theorem}
\label{prop-1.2}
There exists a norm~$\rho$ on~$\BR^2$ such that the minimizer~$U_0$ is not homothetic to the (unique) minimizer of~\eqref{E:1.7w}. In fact, if the norm is smooth outside the origin and~$\partial U_0$ is $C^2$, then the minimizers of \eqref{E:1.13} and \eqref{E:1.7w} are homothetic if and only  if~$U_0$ is a Euclidean ball and~$\rho$ is proportional to the Euclidean norm.
\end{theorem}

One can also wonder about the analytic form of the dependence of the minimizer of~$U_0$ on the underlying norm~$\rho$. Related to this is the question whether, and by what quantitative means, the near minimizers are close to the actual minimizers. We will measure this closeness using the Hausdorff metric 
\begin{equation}
\label{E:Hausdorff}
\distH(A,B):=\max\Bigl\{\sup_{x\in A}\,\inf_{y\in B}\,|x-y|_2,\sup_{x\in B}\,\inf_{y\in A}\,|x-y|_2\Bigr\}
\end{equation}
where~$|\cdot|_2$ will henceforth denote the Euclidean norm on~$\BR^2$. We then have:

\begin{theorem}[Continuity of minimizers in Hausdorff metric]
\label{thm-1}
Recall the notation~$\CM$ from \eqref{E:1.14}.
For each $\epsilon>0$ there is $\delta>0$ such that, for each $U\in\CU$,
\begin{equation}
\label{E:1.15}
\distH(U,\CM)>\epsilon\quad\Rightarrow\quad \CF(U)\ge \min\CF+\delta.
\end{equation}
Moreover, if we center~$U_0$ so that $U_0=-U_0$ and then regard it as an element of the metric space~$(\CU,\distH)$ and~$\rho$ as an element of the space of continuous functions on the unit (Euclidean) sphere in~$\BR^2$ endowed with the supremum norm, then~$U_0$ depends continuously on~$\rho$.
\end{theorem}

We remark that \eqref{E:1.15} is analogous to the Bonnesen inequality known from the context of the isoperimetric problem. (A formulation of this inequality for perimeters defined using a general norm has been given by Dobrushin, Koteck\'y and Shlosman~\cite[Chapter~2]{dks1992}.) The only quantitative connection we know of between the shape and the norm requires passage through the corresponding eigenfunction (see Remarks~\ref{rem1}--\ref{rem2}) and thus remains quite inexplicit.

\subsection{Facets and corners}
Our next item of concern is the appearance of facets and corners on the boundary of~$U_0$. Here a \emph{facet} is a maximal (with respect to inclusion) non-trivial closed linear segment contained in~$\partial U_0$. A \emph{corner} is in turn a point on~$\partial U_0$, regarded as an oriented curve, where the unit tangent vector undergoes a jump. (By convexity of~$U_0$, the left and right unit tangent is well defined at all points of~$\partial U_0$.)  

We will link the facets of~$U_0$ to (what we will call) \emph{degenerate directions} of~$\rho$. To define this concept, fix~$\rho$ and let $e$ and~$e'$ be two orthogonal unit vectors. By the triangle inequality, the function $s\mapsto\rho(e+se')$ is convex and so the one-sided derivatives
\begin{equation}
\label{E:degen}
\theta^\pm:=\frac{\textd}{\textd s^\pm}\rho(e+se')\Bigl|_{s=0}
\end{equation}
are well defined.
As an immediate consequence of the definitions we obtain
\begin{equation}
\label{E:2.37}
-\rho(e')\le\theta^-\le\theta^+\le\rho(e')
\end{equation}
and if we mark the dependence on~$e$ and~$e'$ explicitly as $\theta^\pm(e,e')$, we also have
\begin{equation}
\theta^+(e,-e')=-\theta^-(e,e')=-\theta^-(-e,-e')=\theta^+(-e,e').
\end{equation}
In particular, the difference $\theta^+(e,e')-\theta^-(e,e')$ is preserved under the reversal of the orientation of any (or both) of the vectors~$e$ and~$e'$. We now put forward:

\begin{definition}[Degenerate direction]
Let~$e$ be a unit vector and $e'$ a unit vector orthogonal to~$e$. Given a norm~$\rho$ on~$\BR^2$, we say that~$e$ is \emph{degenerate for~$\rho$} if
\begin{equation}
\theta^+(e,e')-\theta^-(e,e')>0.
\end{equation}
\end{definition}

We then claim:

\begin{theorem}[Facets equivalent to degeneracy]
\label{prop-2.12}
Let~$U_0\in\CU$ be the (unique, convex) minimizing shape of~$U\mapsto\CF(U)$ for~$\CF$ defined using a norm~$\rho$. Let~$e$ be a unit vector in~$\BR^2$. Then the curve constituting $\partial U_0$ contains a non-trivial linear segment in the direction of~$e$ if and only if~$e$ is degenerate for~$\rho$.
\end{theorem}

As is easy to check, the $\ell^1$-norm has $\pm 2^{-1/2}(e_1\pm e_2)$ as its degenerate directions, and this will remain in effect even when we add an arbitrary norm to the $\ell^1$-norm. By adding appropriately weighted rotations of the $\ell^1$-norm, we can thus construct a norm whose minimizing shape will have infinitely many facets.

\smallskip
Let us now move to the appearance of corners. Pick $U\in\CU$ and let $t\mapsto\gamma(t)$ be the parameterization of~$\partial U$ by the \emph{Euclidean} arc-length in the counterclockwise direction. If~$U$ is convex, then~$\gamma$ is right/left differentiable which means that 
\begin{equation}
v^\pm(x):=\frac{\textd}{\,\textd t^\pm}\gamma(t),\quad\text{where}\quad x:=\gamma(t),
\end{equation}
exist with the vectors $\nu^\pm(x)$ having unit Euclidean norm. Note that, since the definition is for a fixed parametrization of the curve, we may and will regard~$v^\pm$ as functions of~$x\in\partial U$.

\begin{definition}
We say that a convex set $U\in\CU$ has a \emph{corner} at a point~$x_0\in\partial U_0$ if
\begin{equation}
v^+(x_0)\ne v^-(x_0)
\end{equation}
\end{definition}
  
We then have:

\begin{theorem}[Corners and strict convexity of the norm]
\label{prop-2.16}
Let~$\rho$ be a norm on~$\BR^2$ and let~$U_0$ be a (convex) minimizer of $U\mapsto\CF(U)$ for the norm~$\rho$. If~$U_0$ has a corner at a point~$x_0\in\partial U_0$, then~$\rho$ and the vectors $v^\pm:=v^\pm(x_0)$ must obey
\begin{equation}
\label{E:2.60}
\rho(a v^-+bv^+)=a\rho(v^-)+b\rho(v^+),\qquad a,b\ge0.
\end{equation}
On the other hand, if \eqref{E:2.60} holds for two distinct unit vectors $\nu^+,\nu^-\in\BR^2$, then there is a point $x_0\in\partial U_0$ such that $\nu^\pm(x_0)=\nu^\pm$ and~$U_0$ thus has a corner at~$x_0$. In particular, $\partial U_0$ is continuously differentiable if and only if~$\rho$ is strictly convex. 
\end{theorem}

By the symmetries of the norm, if~$U_0$ is centered and has a corner at~$x_0$ with tangent vectors~$\nu^\pm$, then it also has a corner at $-x_0$ with tangent vectors $-\nu^\pm$. However, this does not imply that~$U_0$ is a parallelogram as that would require \eqref{E:2.60} for all $a,b\in\BR$. In fact, there are norms for which~$U_0$ has corners and yet no facets at all. This is best seen from the fact that the above characterizations of facets and corners (including the proofs thereof) apply equally well to the isoperimetric problem \eqref{E:1.7w}. There we can enforce the desired features of the minimizing shape~$\widetilde U_0$ from \eqref{E:wulff} directly and then extract $\rho$ by the duality transformation $\rho(x):=\sup\{x\cdot y\colon y\in \widetilde U_0\}$.

\section{Earlier work, key ideas, remarks and outline}
\noindent
Having introduced the problem and stated the main results, let us now discuss briefly the context of this work including the prior literature on the subject. We will then also outline the key ideas and give a list of remarks and directions for possible extensions.

\subsection{Earlier work}
Shape optimization for Dirichlet eigenvalues of the Laplacian has been a topic of considerable interest in recent years. Early work focused on optimizing the eigenvalues under a constraint on the Lebesgue volume. Here the case of the principal eigenvalue goes back to Faber and Krahn's solution of the problem posed by Lord Rayleigh in his ``Theory of Sound.''  The case of higher eigenvalues has been resolved only in the last 10-15 years; see, e.g., the books by Henrot~\cite{henrot} and Velichkov~\cite{Velichkov} for details as well as some historical perspective.

More recently, attempts have been made to find shapes optimizing Dirichlet eigenvalues of the Laplacian subject to other types of constraints than volume. In particular, this includes constraints on the perimeter. The relevant references are the papers of Bucur and Freitas~\cite{Bucur-Freitas}, Bucur, Butazzo and Henrot~\cite{BBH2011}, van den Berg and Inversen~\cite{vdBI2013} and de Philippis and Velichkov~\cite{dPV2014}. However, all of these works focus solely on the perimeter defined using the Euclidean norm and, since this makes the case of the principal eigenvalue directly solvable, on higher-order eigenvalues.

Our interest in perimeters that are defined using non-isotropic norms arises from a study of a random walk on~$\BZ^2$ that (self-)interacts through the size of its boundary. As a result of the interaction, the random walk gets squeezed to a subdiffusive spatial scale. Berestycki and Yadin~\cite{BY2013}, who introduced this problem in the first place, identified the typical diameter of the support of the walk for large times. In~\cite{Biskup-Procaccia}, the present authors show that (in the limit of large times and vanishing temperatures) the asymptotic shape of the visited set is given by minimizing the functional \eqref{E:1.13} for a suitably defined underlying norm~$\rho$. This norm is model dependent, invariably non-isotropic (due to persistent lattice structure) and exhibiting degenerate directions as well as directions where strict convexity fails. Hence our interest in the understanding of the resulting effects on the shape of the minimizer.

\subsection{Main ideas and their limitations}
Let us now move to discuss some aspects of the proofs in the present paper. Some of these are fairly straightforward (e.g., the existence and convexity of a minimizing set) but others require novel and (we think) interesting ideas. One such case is the proof of uniqueness (up to translates) of the minimizing set which is based on a convexity argument with respect to the Minkowski addition of sets: For $U,V\subset\BR^2$ and $\alpha\in[0,1]$, let
\begin{equation}
\alpha U+(1-\alpha)V:=\bigl\{\alpha x+(1-\alpha)y\colon x\in U,\,y\in V\bigr\}.
\end{equation}
For convex~$U$ and~$V$ the principal eigenvalue turns out to be convex,
\begin{equation}
\label{E:2.2ww}
\lambda\bigl(\alpha U+(1-\alpha)V\bigr)\le\alpha\lambda(U)+(1-\alpha)\lambda(V),\qquad\alpha\in[0,1],
\end{equation}
and, in fact, \emph{strictly} so for $\alpha\in(0,1)$ unless the sets~$U$ and~$V$ are homothetic (see~Lemma~\ref{lemma-2.9} whose proof we extract from Colesanti~\cite{colesanti2005brunn}). The perimeter is affine (for convex~$U$ and~$V$),
\begin{equation}
\CP\bigl(\alpha U+(1-\alpha)V\bigr)=\alpha\CP(U)+(1-\alpha)\CP(V),\qquad\alpha\in[0,1],
\end{equation}
(cf Lemma~\ref{lem:parconv}), and so this leads to a type of strict convexity for~$U\mapsto\CF(U)$ that rules out simultaneous occurrence of two non-homothetic minimizers. 

The proof of uniqueness via the stated convexity argument is exactly what forces us to restrict attention to spatial dimension $d=2$. Indeed, the perimeter is no longer affine in~$d\ge3$ and, in fact, the inequality one finds in counterexamples (see Remark~\ref{rem-notaffine}) shows that~$U\mapsto\CF(U)$ is not convex with respect to Minkowski set addition in~$d\ge3$. (In these dimensions, taking the convex hull does not  necessarily decrease the perimeter and so even showing that the minimizer is convex constitutes a non-trivial problem; cf Bucur, Buttazo and Henrot~\cite{BBH2011}.) 

We remark that the results on higher-order eigenvalues (for Euclidean perimeter) in the aforementioned references \cite{Bucur-Freitas,BBH2011,vdBI2013,dPV2014} above do not seem to address uniqueness of a minimizing shape at all. Unfortunately, our method does not appear to contribute in this context either.

Another aspect where our results are limited to $d=2$ is the control of the Hausdorff distance to the minimizing set using the value of~$\CF$. Indeed, similarly to the isoperimetric problem \eqref{E:1.7w}, one can fairly easily find a sequence of sets in $d\ge3$ whose~$\CF$-value approaches the minimum and yet the sets stay a uniform Hausdorff distance away from the class~$\CM$ of the minimizing sets. (Here the $\rho$-perimeter of~$U\in\BR^d$ for~$d\ge3$ is defined as
\begin{equation}
\label{E:2.4ue}
\CP(U):=\int_{\partial U} \rho(n(x))\,H^{d-1}_{\partial U}(\textd x)
\end{equation}
where~$n(x)$ is the normal vector to~$\partial U$ at~$x$ and $H^{d-1}_{\partial U}$ is the $(d-1)$-dimensional Hausdorff measure on~$\partial U$.) A technical reason for the restriction to~$d=2$ is that our proof of Theorem~\ref{thm-1} is conveniently based on a theorem of \v Sver\'ak~\cite{Sverak1993} that provides continuity of the principal eigenvalue under (properly formulated) Hausdorff convergence, but only in $d=2$.

Both the uniqueness of the minimizer and the control of the Hausdorff distance to the minimizers by the value of the functional~$\CF$ are crucial for the application of the above results in the context of the above mentioned random walk problem~\cite{Biskup-Procaccia}. The same applies to our characterization of facets/corners appearing on the minimizing shape. Such features have been of large interest in probabilistic shape theorems (e.g., the recent review by Auffinger, Damron and Hanson~\cite{ADH-review} of shape theorems in first passage percolation). 

\subsection{Further remarks}
\label{sec-remarks}\noindent
We continue with five remarks. The first one addresses certain limitations that were built into the setting of the very problem:

\begin{remark}
Recall that we look for minimizers only in the class of Jordan domains, which are in particular simply connected, connected and with the boundary being the trace of a simple curve. Here are some reasons why we chose this setting.

Simple connectedness is necessary for uniqueness: Removing an isolated point or, in general, a small-enough polar set changes neither the principal eigenvalue nor the perimeter of the set. The restriction to connectedness is more or less an irreducibility statement: The principal eigenvalue of a set with more than one connected components is the minimum of the eigenvalues in the individual components. Dropping all but the minimizing component decreases the perimeter. Finally, the restriction to domains whose boundary is a simple curve is a matter of convenience; particularly, in the definition of the perimeter using the classical arc-length formula~\eqref{E:1.12}.
\end{remark}

The other two remarks are concerned with a direct characterization of the minimizing sets.

\begin{remark}
\label{rem1}
In some special cases (as noted above), the minimizer~$U_0$ can be characterized explicitly. For instance, if~$\rho$ is the Euclidean norm, the Faber-Krahn and isoperimetric inequalities imply that, in the class of sets with a given Lebesgue volume, both the eigenvalue and the perimeter are minimized by the Euclidean ball. By some scaling arguments, the same then applies to the sum thereof. For the $\ell^1$, resp., $\ell^\infty$ norm, the minimizer is an $\ell^\infty$, resp.,~$\ell^1$ ball. This follows by the dual relation of these norms and the monotonicity of the eigenvalue in the underlying set.

Unfortunately, unlike for the isoperimetric problem~\eqref{E:1.7w} where the minimizing shape is simply a ball in the dual norm, for general norm~$\rho$ we do not seem to have a way to describe the minimizing set~$U_0$ directly. For smooth norms with smooth minimizing shape, variational calculus shows that the minimizing shape satisfies
\begin{equation}
\bigl|\nabla h_{U_0}(x)\bigr|^2=C_\rho(x),\qquad x\in\partial U_0,
\end{equation}
where $h_U$ denotes the normalized principal eigenfunction of the Laplacian in~$U$ and~$C_\rho(x)$ is the curvature of~$\partial U_0$ at point~$x$ relative to the arc-length defined by the norm~$\rho$; see the proof of Theorem~\ref{prop-1.2} in Section~\ref{sec-last}. This generally leads to an overdetermined elliptic problem that we do not see prospects of it being explicitly solvable except for one case: $\rho$ =  Euclidean norm.
\end{remark}

\begin{remark}
\label{rem2}
Instead of relying on variational calculus, one might try to mimic the proof that the isoperimetric set for~$\rho$-perimeter is a dilation of the support set of the norm~$\rho$. One is prompted to try this by the fact that the main tool of this proof, the Brunn-Minkowski inequality for the Lebesgue volume, generalizes to the principal eigenvalue of the Laplacian. Indeed, Colesanti~\cite{colesanti2005brunn} showed that for any bounded, open and convex $U,V\subset\BR^d$,
\begin{equation}
\lambda(U+V)^{-1/2}\ge\lambda(U)^{-1/2}+\lambda(V)^{-1/2}.
\end{equation}
This leads to the following inequality for any bounded open and convex $U,V\subset\BR^d$:
\begin{multline}
\qquad
\frac1{\lambda(V)^{3/2}}
\int_{\partial V}\tau_U\bigl(n_V(x)\bigr)\bigl|\nabla h_V(x)\bigr|^2 H_{\partial V}^{d-1}(\textd x)
\\
\ge
\frac1{\lambda(U)^{3/2}}
\int_{\partial U}\tau_U\bigl(n_U(x)\bigr)\bigl|\nabla h_U(x)\bigr|^2 H_{\partial U}^{d-1}(\textd x),
\qquad
\end{multline}
where $n_U(x)$ denotes the normal to $\partial U$ at point~$x$, $\tau_U(n):=\inf\{x\cdot n\colon x\in U\}$ and~$H_{\partial U}^{d-1}$ is the $(d-1)$-dimensional Hausdorff measure on~$\partial U$. The point where this approach currently runs into a dead end is the appearance of the eigenfunctions~$h_V$ and~$h_U$, which seem to prevent us from interpreting the integral as a perimeter functional in \eqref{E:2.4ue}.
\end{remark}

Finally, we want to make the following notes on the qualitative nature of our conclusions:

\begin{remark}
Our proof of Theorem~\ref{prop-2.12} links degenerate directions to facets. However, the argument is not quantitive and so we do not know how the size of the facet depends on the size of $\theta^+-\theta^-$. Progress on this questions may be of interest.
\end{remark}

\begin{remark}
Our proof of Theorem~\ref{thm-1} goes through a limit argument and thus gives no specific relation between~$\epsilon$ and~$\delta$. It would be interesting to derive a quantitative (Bonnesen-like) inequality between $\distH(U,\CM)$ and some function of $\CF(U)-\CF(U_0)$.
\end{remark}

\subsection{Outline}
The remainder of this paper is devoted to the proof of the above results. The strategy is as follows: In Section~\ref{sec2} we first establish the existence of a convex minimizer and then use the aforementioned argument to show that the convex minimizer is unique up to spatial shifts. This proves Theorem~\ref{thm-0}. We then also control the continuity in the underlying norm and rule out non-convex minimizers (and thus prove Theorem~\ref{thm-1}).

In Section~\ref{sec4} we address the criteria for existence/absence of facets (proving Theorem~\ref{prop-2.12}) and relate appearance of corners to non-strict convexity of the norm (proving Theorem~\ref{prop-2.16}). In addition, we also give an explicit example of a norm for which the isoperimetric shape differs from the minimizer of~$\CF$, and prove that these are generically different in the class of smooth norms with smooth minimizers (proving Theorem~\ref{prop-1.2}).

\section{Existence and uniqueness of minimizers}
\label{sec2}\noindent
We are now ready to begin the proofs of our main results. Here we will show all needed facts underlying the proofs of Theorem~\ref{thm-0} and~\ref{thm-1}.

\subsection{Reduction to convex sets}
Recall the notation~$\CF(U)$ for the function in \eqref{E:1.13} and~$\CU$ for the class of admissible domains in \eqref{E:1.11}. The key point of this subsection is the observation that, as far as finding a minimizer is concerned, we can safely restrict attention to convex sets:

\begin{proposition}
\label{prop-2}
Let~$\rho$ be any norm on~$\BR^2$. Given~$U\in\CU$, let $\wh U$ denote the convex hull of~$U$. Then $\wh U\in\CU$ and $\CF(\wh U)\le\CF(U)$.
\end{proposition}

We begin with a lemma that gives \emph{a priori} estimates on the volume and the perimeter of sets we need to consider. These will be used later in our proofs.

\begin{lemma}
\label{lemma-2.1}
Let~$\rho$ be an arbitrary norm on~$\BR^2$. Then
\begin{equation}
\inf\{\CF(U)\colon U\in\CU\}>0
\end{equation}
and there are constants $c_1,c_2,c_3,c_4\in(0,\infty)$ such that the infimum problem is not affected by restricting to sets satisfying
\begin{equation}
\label{E:2.2}
c_1\le|U|\le c_2\quad\text{and}\quad c_3\le |\partial U|\le c_4,
\end{equation}
where $|\partial U|$ denotes the perimeter of~$U$ in the Euclidean norm on~$\BR^2$.
\end{lemma}

\begin{proofsect}{Proof}
Let $\rho_\mini$, resp., $\rho_\maxi$ denote the minimal, resp., maximal value of $\rho(x)$ for~$x$ on the unit Euclidean sphere in~$\BR^2$. Then
\begin{equation}
\rho_\mini|\partial U|\le\CP(U)\le\rho_\maxi|\partial U|,
\end{equation}
so it suffices to prove the lemma just for the Euclidean norm. Here the $2$-dimensional Faber-Krahn inequality states
\begin{equation}
\lambda(U)\ge \frac{c}{|U|},
\end{equation}
where $c$ is the value of $\lambda(U)$ for $U$ being a disc of unit Lebesgue area, while the isoperimetric inequality yields
\begin{equation}
|\partial U|\ge\,c'|U|^{1/2},
\end{equation}
where $c'$ is the Euclidean perimeter of a disc of the unit area.
Putting these bounds together, an upper bound on~$\CF(U)$ thus yields both upper and lower bounds on~$|U|$. The above then extend these to upper and lower bounds on~$|\partial U|$ as well.
\end{proofsect}

The main step of the proof of Proposition~\ref{prop-2} is that fact that (in~$d=2$), taking the convex hull does not increase the perimeter of the set. To prove this fact we have to show that every polygonal approximation of the boundary of the convex hull~$\widehat U$ of a set~$U\in\CU$ can be reduced to a polygonal approximation of~$\partial U$ of at least that much length. Although this appears intuitively obvious, a formal proof requires some geometric details of how~$U$ is related to~$\widehat U$. Instead of bundling these with the actual argument, we make them the content of a separate lemma.

\begin{lemma}
\label{lemma-2.3}
Let $U\in\CU$ and let $\wh U$ be the convex hull of~$U$. Then $\wh U$ is open and, in fact,
\begin{equation}
\label{E:2.6ww}
\wh U = \text{\rm interior of the convex hull of~$\partial U$}.
\end{equation}
In particular, $\wh U$ is connected.
Moreover, if $A$ denotes the set of extremal points of the closure of~${\wh U}$, then $A\subset \partial U$ while $\partial \wh U\smallsetminus A$ is the disjoint union of open linear segments in~$\BR^2$ with both endpoints in~$A$. Furthermore, if $\gamma\colon[0,1]\to\BR^2$ is a simple closed curve parametrizing $\partial U$, then $I_A:=\{t\in[0,1]\colon\gamma(t)\in A\}$ is a closed subset of~$[0,1]$ and
\begin{equation}
\wh\gamma(t):=\begin{cases}
\gamma(t),\qquad&\text{\rm if }t\in I_A,
\\
\text{\rm linear},\qquad&\text{\rm else},
\end{cases}
\end{equation}
is a (continuous) simple closed curve such that $\wh\gamma([0,1])=\partial \wh U$. In particular, $\widehat U\in\CU$.
\end{lemma}

\begin{proofsect}{Proof}
The fact that $\wh U$ is open and connected follows directly from the definitions (and the fact that~$U$ is open). To get \eqref{E:2.6ww} we note that if~$x\in\wh U$ is written as $x=ty+(1-t)z$ for some $y,z\in U$ and~$t\in[0,1]$, then by extending the linear segment $[y,z]$ to the first intersections with~$\partial U$ we can write~$x$ the same way but now with $y,z\in\partial U$.

Let $\text{cl}(B)$ denote the closure of a set~$B\subset\BR^2$ and let $A$ denote the set of extremal points of~$\text{cl}(\wh U)$. Pick $x\in A$ and note that~$x\in\partial\wh U$ implies, by the fact that $U$ is open and $U\subset\wh U$, that~$x\notin U$. But any point in~$\text{cl}(\wh U)$ is in the convex hull of~$\text{cl}(U)$ and so $x\not\in\partial U$ would imply that~$x$ is a non-trivial convex combination of distinct elements~$y,z\in\text{cl}(U)$. 

To finish the proof, note that any $x\in\partial\wh U\smallsetminus A$ must be written as a non-trivial convex combination of \emph{exactly two} distinct elements from~$A$ (otherwise~$x$ would be in the interior of~$\wh U$) and so it lies in an open linear segment connecting these two points. By continuity, $I_A$ is thus closed and~$\wh\gamma$ defined above parametrizes~$\partial\wh U$. The fact that~$\wh\gamma$ is simple is a consequence of convexity of~$\wh U$.
\end{proofsect}

This now permits us to show:

\begin{lemma}
\label{lemma-2.4}
Let $U\in\CU$ and let $\wh U$ be the convex hull of~$U$. Then, for the perimeter defined using any norm on~$\BR^2$, we have $\CP(\wh U)\le\CP(U)$.   
\end{lemma}

\begin{proofsect}{Proof}
Let $\gamma$ be a curve parametrizing~$\partial U$ and let the set~$A$ and curve $\wh\gamma$ be as in Lemma~\ref{lemma-2.3}. First we claim that $\CP(\wh U)$ can be computed by considering only partitions such that $\gamma(t_i)\in A$ for each~$i$. Indeed, consider a partition such that~$\wh\gamma(t_i)$ is not extremal in $\text{cl}(U)$ for some~$i$. Then, by Lemma~\ref{lemma-2.3},~$\wh\gamma(t_i)$ lies in the interior of a linear segment on~$\partial \wh U$. Adding the end-points of this segment to the partition does not decrease the sum in \eqref{E:1.12} but once that is done, dropping the point $\wh\gamma(t_i)$ leaves the sum unchanged. Proceeding inductively, every partition can be turned into a partition on extremal points without decreasing the sum in \eqref{E:1.12}. 

Now pick $\epsilon>0$ small and $M>0$ large and let $\{t_i\}$ be a partition such that 
\begin{equation}
\sum_{i=1}^n\rho\bigl(\wh\gamma(t_i)-\wh\gamma(t_{i-1})\bigr)\ge M\wedge\bigl(\CP(\wh U)-\epsilon).
\end{equation}
By the above argument, we may assume that $\wh\gamma(t_i)\in A$ for all $i$. But then $\gamma(t_i)=\wh\gamma(t_i)$ by the definition of~$\wh\gamma$ and so the sum on the left is no more than~$\CP(U)$. Hence, $\CP(\wh U)\le\CP(U)$.
\end{proofsect}

\begin{proofsect}{Proof of Proposition~\ref{prop-2}}
By the variational characterization of the principal eigenvalue we have $\lambda(\wh U)\le\lambda(U)$. The claim then follows from Lemma~\ref{lemma-2.4}.
\end{proofsect}

\subsection{Minimizer among convex sets: existence}
We will now proceed to prove the existence of a minimizer of~$U\mapsto\CF(U)$ in the class of convex sets. The arguments in this section are quite standard --- we basically just need to prove lower semicontinuity of~$U\mapsto\CF(U)$ over the class of convex sets. Still, we provide detailed proofs as we do not find all of them satisfactorily worked out in the literature.

\begin{proposition}
\label{lem:exist}
There exists a convex set $U\in\CU$ such that $\CF(U)=\inf\{\CF(U')\colon U'\in\CU\}$. 
\end{proposition}

We begin with a simple observation, whose proof is probably well known:

\begin{lemma}
\label{lemma-3.6u}
Let $U,V\in\CU$ be convex and (recalling that $0\in U,V$) let $\dist_2(0,\partial U)$ denote the Euclidean distance of~$0$ to~$\partial U$. Then
\begin{equation}
\label{E:3.9ue}
(1-\epsilon)U\subseteq V\subseteq(1+\epsilon)U
\end{equation}
holds for any $\epsilon>\distH(U,V)/\dist_2(0,\partial U)$ with~$\epsilon<1$.
\end{lemma}

\begin{proofsect}{Proof}
For any~$r>0$, let $B(r)$ denote the open Euclidean ball of radius~$r$ centered at~$0$ and pick~$\delta>\distH(U,V)$. The definition of the Hausdorff distance \eqref{E:Hausdorff} implies
\begin{equation}
\label{E:3.10ue}
V\subseteq U+B(\delta)\quad\text{and}\quad U\subseteq V+B(\delta).
\end{equation}
The inequality on the right is then nearly immediate: Abbreviate $r_U:=\dist_2(0,\partial U)$. Then
\begin{equation}
V\subseteq U+B(\delta)= U+\delta r_U^{-1} B(r_U)\subseteq U+\delta r_U^{-1} U,
\end{equation}
where we used that $B(r_U)\subseteq U$.

For the inequality on the left, we will first prove that, for any $\delta>0$,
\begin{equation}
\label{E:3.11ue}
r_{V+B(\delta)}\le r_V+\delta.
\end{equation}
(In fact, equality holds here and is easy to prove, but inconsequential for us.) Indeed, let~$x_0$ be the point on~$\partial V$ such that~$r_V$ equals the Euclidean norm~$|x_0|$. Set $\BH:=\{x\in\BR^2\colon x\cdot x_0<r_V^2\}$. The convexity of~$V$ then forces $V\subseteq\BH$. Hence $V+B(\delta)\subseteq\BH+B(\delta)$ and so $r_{V+B(\delta)}$ is at most the distance of~$0$ to the boundary of $\BH+B(\delta)$ and \eqref{E:3.11ue} thus holds.

To finish the proof of the main claim, let again $\delta>\distH(U,V)$ but now assume also $\delta<r_U$. (This is possible else there is no~$\epsilon$ to apply the statement to.) By the inclusion on the right of~\eqref{E:3.9ue} (which we already proved above),
\begin{equation}
\label{E:3.12ue}
U\subseteq (1+\delta r_V^{-1})V.
\end{equation}
But the second inclusion in \eqref{E:3.10ue} gives $r_U\le r_{V+B(\delta)}$ and from \eqref{E:3.11ue} we then get $r_U\le r_V+\delta$. This yields
\begin{equation}
1+\frac{\delta}{r_V}\le1+\frac{\delta}{r_U-\delta}=\frac1{1-\frac\delta{r_U}}
\end{equation}
and, by a simple rescaling of \eqref{E:3.12ue}, we then get the inclusion on the left \eqref{E:3.9ue} as well.
\end{proofsect}

We remark that \eqref{E:3.11ue} is easily invalidated for~$U$ or $V$ non-convex. As a consequence of Lemma~\ref{lemma-3.6u}, we get:

\begin{lemma}
\label{lemma-3.7vv}
Let $\{U_n\}\subset\CU$ and $U\in\CU$ be convex sets. Then
\begin{equation}
\distH(U_n,U)\,\underset{n\to\infty}\longrightarrow\,0\quad\Rightarrow\quad
\lim_{n\to\infty}\lambda(U_n)=\lambda(U).
\end{equation}
\end{lemma}

\begin{proofsect}{Proof}
This statement is well known, e.g., Henrot~\cite[Theorem~2.3.15]{henrot}, but we provide a proof as it is easy and instructive: Lemma~\ref{lemma-3.6u} implies that, for each~$\epsilon>0$ small enough,
\begin{equation}
(1-\epsilon)U\subset U_n\subset(1+\epsilon)U
\end{equation}
once~$n$ is large. Using this in \eqref{E:1.4} along with the scaling relation \eqref{E:scaling} shows~$\lambda(U_n)\to\lambda(U)$.
\end{proofsect}

For the perimeter we only get a statement of lower semicontinuity with respect to Hausdorff convergence (on convex sets). For the sake of later applications, we will also allow the underlying norm to vary with~$n$.

\begin{lemma}
\label{lemma-3.7u}
Let $\rho_n$ be a sequence of norms converging, if regarded as continuous functions on the unit circle in~$\BR^2$, uniformly to a norm~$\rho$. Let~$\CP_n$, resp.,~$\CP$ denote the perimeter with respect to~$\rho_n$, resp.,~$\rho$, and let $\{U_n\}\subset\CU$ and $U\in\CU$ be convex sets. Then
\begin{equation}
\label{E:3.17uq}
\distH(U_n,U)\,\underset{n\to\infty}\longrightarrow\,0\quad\Rightarrow\quad
\CP(U)\le\liminf_{n\to\infty}\CP_n(U_n).
\end{equation}
Moreover, for any $U\in\CU$ convex,
\begin{equation}
\label{E:3.18uq}
\lim_{n\to\infty}\CP_n(U) = \CP(U).
\end{equation}
\end{lemma}

\begin{proofsect}{Proof}
Let $\gamma\colon[0,1]\to\BR^2$ be the curve parametrizing~$\partial U$ using the polar angle (in units of $2\pi$) as viewed from~$0$, and let $\gamma_n$ be the corresponding parametrization of~$\partial U_n$. Pick a collection of points $0=t_0<t_1<\dots<t_m=1$. Lemma~\ref{lemma-3.6u} shows
\begin{equation}
\max_{i=0,\dots,m}\,\rho_n\bigl(\gamma_n(t_i)-\gamma(t_i)\bigr)\le\frac{2c}r\distH(U_n,U),
\end{equation}
for any~$r$ with $0<r<\dist_2(0,\partial U)$, where~$c:=\sup_{n\ge1}\sup\{\rho_n(x)\colon |x|_2=1\}$. Hence, the convergence $\distH(U_n,U)\to0$ implies
\begin{equation}
\liminf_{n\to\infty}\CP_n(U_n)\ge\liminf_{n\to\infty}\,\sum_{i=1}^m\rho_n\bigl(\gamma_n(t_i)-\gamma_n(t_{i-1})\bigr)
=\sum_{i=1}^m\rho\bigl(\gamma(t_i)-\gamma(t_{i-1})\bigr).
\end{equation}
Taking supremum over $\{t_i\}$ then yields \eqref{E:3.17uq}.

For the second part of the claim, the convergence of~$\rho_n$ to~$\rho$ on the unit circle is actually uniform which after a moment's though gives that, for each~$\epsilon>0$,
\begin{equation}
(1-\epsilon)\rho(x)\le\rho_n(x)\le(1+\epsilon)\rho(x),\qquad x\in\BR^2,
\end{equation}
holds as soon as~$n$ is sufficiently large. This implies
\begin{equation}
(1-\epsilon)\CP(U)\le\CP_n(U)\le(1+\epsilon)\CP(U)
\end{equation}
for~$n$ large thus proving also \eqref{E:3.18uq}.
\end{proofsect}

\begin{proofsect}{Proof of Proposition~\ref{lem:exist}}
Let $\{U_n\}\subset\CU$ be a minimizing sequence, i.e.,
\begin{equation}
\CF(U_n)\,\underset{n\to\infty}\longrightarrow\,\inf\{\CF(U)\colon U\in\CU\}
\end{equation}
 By Lemma~\ref{lemma-2.1} we may assume that~$\{U_n\}$ obey the bounds in \eqref{E:2.2} and, in particular, are uniformly bounded. Moreover, by Proposition~\ref{prop-2} we may assume that~$U_n$ are convex. By Lemma~\ref{lemma-2.1} there is an $r>0$ such that each $U_n$ contains an open disc of radius $r$ (otherwise~$|\partial U_n|$ is larger than order $|U|/r$). Applying a suitable shift to each~$U_n$ we may thus assume $\dist_2(0,\partial U_n)\ge r$ for each~$n$.

Reducing to a subsequence if necessary, the Blaschke Selection Theorem (or a direct argument using the polar-angle parameterization used in the proof of Lemma~\ref{lemma-3.7u}) permits us to assume that~$\{U_n\}$ converge in the Hausdorff distance to an open (by our choice) set~$U$. Clearly,~$U$ is also convex and thus connected and, by the above centering assumption, $0\in U$. Lemmas~\ref{lemma-3.7vv}--\ref{lemma-3.7u} then give $\CF(U)\le\liminf_{n\to\infty}\CF(U_n)$. Hence, $U$ is a minimizer.
\end{proofsect}

\subsection{Minimizer among convex sets: uniqueness}
Since we have already proved the existence of a minimizer, we will henceforth write $\min\CF$ for the infimum in the statement of Proposition~\ref{lem:exist}. Our next task is to show that the convex minimizer is actually unique modulo translation:

\begin{proposition}
\label{prop-3}
There is a convex set $U_0\in\CU$ such that
\begin{equation}
\bigl\{U\in\CU\colon \text{\rm convex},\,\,\CF(U)=\min\CF\bigr\}=\bigl\{x+U_0\colon -x\in U_0\bigr\}
\end{equation}
\end{proposition}

As mentioned earlier, the argument will be based on convexity, but this time with respect to convex combinations based on Minkowski addition of sets. For the eigenvalue, we get:

\begin{lemma}
\label{lemma-2.9}
For any bounded non-empty open convex sets $U,V\subset\BR^2$ and any $\alpha\in[0,1]$,
\begin{equation}
\lambda\bigl(\alpha U+(1-\alpha)V\bigr)\le\alpha\lambda(U)+(1-\alpha)\lambda(V).
\end{equation}
Moreover, assuming~$\alpha\in(0,1)$, equality holds only if~$U$ and~$V$ are homothetic.
\end{lemma}

\begin{proofsect}{Proof}
The argument can be gleaned from the proof of Theorem~1 in Colesanti~\cite{colesanti2005brunn}. Indeed, convexity of~$\lambda$ is shown in formula~(45) on page 127 while strict convexity for non-homothetic~$U$ and~$V$ follows by assuming equality in formula~(45) and then following the argument starting from the second display from the bottom on page 128.
\end{proofsect}

Concerning the perimeter, we claim:

\begin{lemma}
\label{lem:parconv}
$\mathfs{P}(\cdot)$ is an affine function on the class of bounded, open and convex sets. Explicitly, for any $U,V\subset\BR^2$ bounded, open and convex,
\begin{equation}
\CP\bigl(\alpha U+(1-\alpha)V\bigr)=\alpha\CP(U)+(1-\alpha)\CP(V),\qquad \alpha\in[0,1].
\end{equation}
\end{lemma}

\begin{proof}
Let $U$ and $V$ be bounded, open and convex sets in $\BR^2$. Since $\alpha\CP(U)=\CP(\alpha U)$, it suffices to prove
\begin{equation}
\label{E:2.15}
\CP(U+V)=\CP(U)+\CP(V)
\end{equation}
for $U+V:=\{u+v\colon u\in U,\,v\in V\}$. The setting is illustrated in Fig.~\ref{fig1}.

\newcounter{obrazek}

\begin{figure}[t]
\refstepcounter{obrazek}
\label{fig1}
\vspace{.2in}
\includegraphics[width=2.8in]{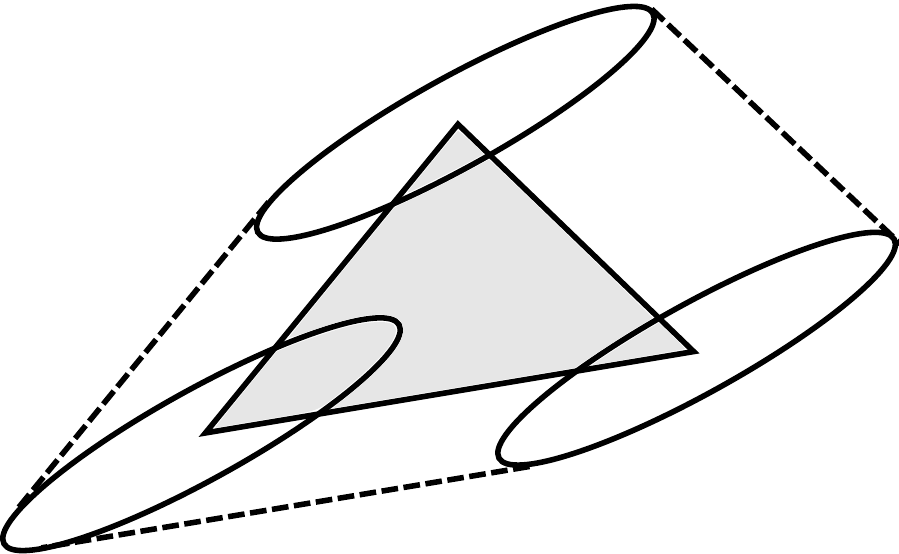}
\vspace{.2in}
\bigskip
\begin{quote}
\fontsize{9}{5}\selectfont
{\sc Fig.~\theobrazek.\ } An illustration of the argument underpinning  the proof of Lemma~\ref{lem:parconv}. Here the shaded triangle corresponds to the convex polygon~$U$ while the ellipses stand for the set~$V$. Note that the dashed lines are just translates of the sides of the triangle.
\normalsize
\end{quote}
\end{figure}

Let us first assume that~$U$ is a convex polygon. Denote by $u_1,\dots,u_n$ the vertices of~$U$ listed in the clockwise direction and let $e_1:=(u_1,u_2),\ldots,e_n:=(u_n,u_1)$ be the corresponding edges regarded as vectors in~$\BR^2$. For every $1\le i\le n$, there is at least one point $v_i\in\partial V$ for which
\begin{equation}
\bigl\{sv_i+t{e}_i\colon s\ge1,\,t\in\BR\}\cap V=\emptyset.
\end{equation}
The points $\{v_i\colon i=1,\dots,n\}$ are not necessarily unique but the above ensures that they run around~$\partial V$ in the clockwise direction.
 
Now consider the set $C:=U+V$. We claim that $C$ is the convex hull of $\bigcup_{i=1}^n(u_i+V)$. Indeed, every point in~$C$ can be written as $u+v$ for some $u\in U$ and $v\in V$. Since there are $t_1,\ldots,t_n\ge0$ satisfying $\sum t_i=1$ and $\sum t_i u_i=a$, we have
\begin{equation}
u+v=\sum_{i=1}^nt_i(u_i+v).
\end{equation}
It follows that $C$ is a subset of the convex hull of $\bigcup_{i=1}^n(u_i+V)$; the other inclusion is trivial. 

Now by the definition of $v_i$ and the convexity of~$U$ we obtain that the line connecting $u_i+v_i$ and $u_{(i+1)\text{ mod } n}+v_i$ is in $\partial C$. The $\rho$-lengths of these segments add up to exactly $\CP(U)$. On the other hand, the complement of these segments in~$\partial C$ is the union of closed arcs connecting~$v_i$ to $v_{(i+1)\text{ mod }n}$ on~$\partial V$. The sum of the~$\rho$-lengths of these arcs is~$\CP(V)$. Hence $\CP(C)=\CP(U)+\CP(V)$ thus proving \eqref{E:2.15} when~$U$ is a polygon.

Let now~$U,V$ be general bounded, open and convex sets. Given~$\epsilon>0$, let $U_\epsilon$ be a convex polygon with vertices on~$\partial U$ such that
\begin{equation}
\label{E:3.26ue}
\distH(U,U_\epsilon)<\epsilon \quad\text{and}\quad\CP(U)\ge\CP(U_\epsilon)\ge\CP(U)-\epsilon,
\end{equation}
and let $V_\epsilon$ be the analogous object for~$V$. Then $U_\epsilon+V_\epsilon$ is a convex polygon whose vertices lie on~$\partial(U+V)$ and so, by the definition of the perimeter and \eqref{E:2.15} for polygonal domains,
\begin{equation}
\CP(U+V)\ge\CP(U_\epsilon+V_\epsilon)=\CP(U_\epsilon)+\CP(V_\epsilon)\ge\CP(U)+\CP(V)-2\epsilon.
\end{equation}
Taking $\epsilon\downarrow0$, we get $\ge$ in \eqref{E:2.15} for~$U$ and~$V$. To get also $\le$, we use that $U_\epsilon\to U$ and $V_\epsilon\to V$ in the Hausdorff distance. Lemma~\ref{lemma-3.7u} then shows
\begin{equation}
\begin{aligned}
\CP(U+V)
&\le \liminf_{\epsilon\downarrow0}
\CP(U_\epsilon+V_\epsilon)
\\&=\liminf_{\epsilon\downarrow0}\,\bigl(\CP(U_\epsilon)+\CP(V_\epsilon)\bigr)
\le \CP(U)+\CP(V),
\end{aligned}
\end{equation}
where we used the claim for convex polygons to get the equality and \eqref{E:3.26ue} for both~$U$ and~$V$ to get the inequality on the extreme right.
This proves \eqref{E:2.15} and thus the whole claim.
\end{proof}

\begin{remark}
\label{rem-notaffine}
It is easy to come up with examples of non-convex sets in~$\BR^2$ for which the statement of Lemma~\ref{lem:parconv} fails. Similarly, the example of $U:=(0,1)^d$ and $V:=$ Euclidean ball shows that this lemma does not apply in~$d\ge3$ either.
\end{remark}

We are now ready to prove the above uniqueness claim:

\begin{proofsect}{Proof of Proposition~\ref{prop-3}}
Suppose that $U,V\in\CU$ are two convex minimizers of~$\CF$. If $U$ and~$V$ were not homothetic, then the combination of Lemmas~\ref{lem:parconv} and~\ref{lemma-2.9} would show 
\begin{equation}
\CF\bigl(\alpha U+(1-\alpha)V\bigr)<\alpha \CF(U)+(1-\alpha)\CF(V),\qquad \alpha\in(0,1)
\end{equation}
and so at least one of~$U$ and~$V$ could not be a minimizer. Hence,~$U$ and~$V$ are homothetic and so there is~$x\in\BR^2$ and $t>0$ such that $U=x+tV$. 
The scaling relations \eqref{E:scaling} then imply $\lambda(U)=t^{-2}\lambda(V)$ and $\mathfs{P}(U)=t\mathfs{P}(V)$ and so $\CF(U)=g(t)$ where
\begin{equation}
g(s):=\CF(sV)=s^{-2}\lambda(V)+s\mathfs{P}(V),\qquad s>0.
\end{equation}
Now~$g$ is strictly convex with a unique minimum on $(0,\infty)$. Since $g(1)=\CF(V)$ while $g(t)=\CF(U)$, if both~$U$ and~$V$ are to be minimizers of~$\CF$ then we must have~$t=1$. We conclude that~$U$ is a translate of~$V$, thus proving the statement.
\end{proofsect}

\subsection{Ruling out non-convex minimizers}
The next task is to show that there are in fact no non-convex minimizers of~$U\mapsto\CF(U)$. For this purpose, as well as further use later, we will need:

\begin{lemma}
\label{lemma-hull}
Suppose $U\subset\BR^2$ be open and simply connected and let $U'\subset\BR^2$ be bounded, open and convex with $U\subseteq U'$. Then $U\ne U'$ implies
$\lambda(U')<\lambda(U)$.
\end{lemma}

\begin{proofsect}{Proof}
Since $U$ is simply connected, the fact that $U\ne U'$ implies that $U'\smallsetminus U$ contains a closed connected set of non-zero diameter. In $d=2$, this set has a non-zero capacity in $U'$, i.e., $\text{cap}_{U'}(U'\smallsetminus U)>0$. (Here for $V\subseteq U$ bounded and $V$ relatively compact in~$U$, $\text{cap}_U(V)$ is the minimum of~$\Vert\nabla f\Vert_2^2$ over $C^2$-functions~$f$ such that~$f:=1$ on~$V$ and~$f:=0$ on~$U^\cc$.) By known facts about Dirichlet eigenvalues (see, e.g., Henrot~\cite[Section~1.3.2]{henrot}), if $V\subseteq U$ with~$\text{cap}_U(U\smallsetminus V)>0$ then $\lambda(V)>\lambda(U)$.
\end{proofsect}

We can now give:

\begin{proofsect}{Proof of Theorem~\ref{thm-0}}
Let $\rho$ be a norm on $\BR^2$. By Lemmas~\ref{lemma-2.4} and~\ref{lemma-hull}, every minimizer is convex while Proposition~\ref{prop-3} ensures that the convex minimizer is unique up to translations. It follows that the set $\CM$ in \eqref{E:1.14} for $U_0$ as in Proposition~\ref{prop-3} exhausts all minimizers of~$\CF$ on~$\CU$.

As is easily checked, $U\mapsto\lambda(U)$ is preserved by rotations, reflections and translations of~$\BR^2$ while $U\mapsto\CP(U)$ is preserved by translations, reflections and exactly those rotations that preserve~$\rho$. The uniqueness of the minimizer up to translations implies that also the minimizing shape is preserved by those symmetries. 
\end{proofsect}

The above also permits us to prove the continuity statement in Theorem~\ref{thm-1}:

\begin{proofsect}{Proof of Theorem~\ref{thm-1} (continuity of $\rho\mapsto U_0$)}
Let~$\rho_n$ be a sequence of norms converging to~$\rho$ uniformly on the unit (Euclidean) sphere. Let~$\CF_n$, resp., $\CF$ denote the functional of interest for norms $\rho_n$, resp., $\rho$ and  let~$U_0^n$, resp.,~$U_0$ denote the corresponding (unique) centered minimizers. Reducing to a subsequence if needed, we may assume that $U_0^n$ converge in the Hausdorff distance to an open convex set~$V$. Lemmas~\ref{lemma-3.7vv}-\ref{lemma-3.7u} then give
\begin{equation}
\CF(V)\le\liminf_{n\to\infty}\CF_n(U_0^n).
\end{equation}
Our task is to show that~$V=U_0$. Suppose, for the sake of contradiction, that~$V\ne U_0$. Then the fact that~$V$ is also centered means that $\CF(V)>\CF(U_0)$. But Lemma~\ref{lemma-3.7u} also ensures that
\begin{equation}
\CF(U_0)=\lim_{n\to\infty}\CF_n(U_0)
\end{equation}
and so if~$\CF(U_0)<\CF(V)$, then $\CF_n(U_0^n)>\CF_n(U_0)$ for all~$n$ sufficiently large. This contradicts the minimality of~$U_0^n$ and so~$\CF(V)=\CF(U_0)$ after all. The uniqueness of the minimizer modulo shifts then yields $U_0^n\to U_0$ in the Hausdorff distance.
\end{proofsect}

To finish the proof Theorem~\ref{thm-1} we first show:

\begin{proposition}
\label{prop-2.10}
For each~$\epsilon>0$ there is $\delta>0$ such that for all~$U\in\CU$ and with~$\wh U$ denoting the convex hull of~$U$,
\begin{equation}
\distH(U,\wh U)>\epsilon\quad\Rightarrow\quad\CF(U)\ge\min\CF+\delta.
\end{equation}
\end{proposition}

\begin{proof}
Suppose, for the sake of contradiction, that the statement fails. Then there is $\epsilon>0$ and a sequence of sets $V_n\in\CU$ such that
\begin{equation}
\label{E:2.26}
\distH(V_n,\wh V_n)>\epsilon\quad\text{but}\quad \lim_{n\rightarrow\infty}\CF(V_n)=\min\CF.
\end{equation}
By Lemma~\ref{lemma-2.1}, all $V_n$ lie inside a large closed ball~$B$ centered at the origin. The same lemma also ensures that all $V_n$'s contain an open disc of some radius $r>0$ (otherwise, as noted before, $|\partial V_n|\ge c|V_n|r^{-1}$ for some constant~$c>0$ depending only on~$\rho$). Thanks to shift invariance of~$\CF$, we may thus assume that every~$V_n$ contains a disc of radius~$r$ centered at~$0$. (This ensures that the containment of zero in the interior is preserved under various Hausdorff limits below.)

Lemma~\ref{lemma-2.4} gives $\CF(\widehat{V}_n)\le \CF(V_n)$ and thus also $\CF(\widehat{V}_n)\to\min\CF$. Reducing to a subsequence of~$\{V_n\}$ if necessary, the same argument as used in the proof Proposition~\ref{lem:exist} implies that, for some convex $W\in\CU$,
\begin{equation}
\label{E:2.27w}
\distH(\widehat{V}_{n},W)\,\,\underset{n\to\infty}\longrightarrow\,\, 0
\end{equation}
and
\begin{equation}
\label{E:2.28}
\CP(\wh V_{n})\,\underset{n\to\infty}\longrightarrow\,\CP(W)
\quad\text{and}\quad
\CF(\widehat{V}_{n})\,\,\underset{n\to\infty}\longrightarrow\,\,\CF(W).
\end{equation}
In particular, $\CF(W)=\min\CF$ and so $W$ is a convex minimizer.

By sequential compactness of the class of closed subsets of~$B$ in the Hausdorff distance, we may (by reducing to another subsequence) assume that,  for some open $V\subset B$,
\begin{equation}
\label{E:3.35uw}
\distH\bigl(B\smallsetminus V_n, B\smallsetminus V\bigr)\,\underset{n\to\infty}\longrightarrow\,0.
\end{equation}
Although this does not imply that $V_n$ converges to~$V$ in the Hausdorff distance, since $V_n\subseteq\wh V_n\subset B$ and $\wh V_n\to W$ in the Hausdorff distance, we definitely get
\begin{equation}
B\smallsetminus V\supseteq B\smallsetminus W\quad\text{and so}\quad V\subseteq W.
\end{equation}
Moreover, the inequality on the left of \eqref{E:2.26} shows
\begin{equation}
\begin{aligned}
\distH(V,W)&=\distH(B\smallsetminus V,B\smallsetminus W)
\\
&=\lim_{n\to\infty}\distH(B\smallsetminus V_n,B\smallsetminus\wh V_n)
\\
&=\lim_{n\to\infty}\distH(V_n,\wh V_n)\ge\epsilon
\end{aligned}
\end{equation}
and so $V$ is in fact a strict subset of~$W$.

We claim that this leads to a contradiction. Indeed, the sets $V_n$ are simply connected and, although not necessarily connected, so is~$V$. By a theorem of \v Sver\'ak~\cite{Sverak1993} (transcribed in Henrot~\cite[Theorem~2.3.19]{henrot}), \eqref{E:3.35uw} then implies
\begin{equation}
\lambda(V_n)\,\underset{n\to\infty}\longrightarrow\,\lambda(V).
\end{equation}
Lemma~\ref{lemma-2.4} in turn  shows
\begin{equation}
\lambda(V_n)=\CF(V_n)-\CP(V_n)\le\CF(V_n)-\CP(\widehat V_n)
\end{equation}
and so \eqref{E:2.28} gives
\begin{equation}
\label{E:2.33}
\lambda(V)\le\CF(W)-\CP(W)=\lambda(W).
\end{equation}
But $V$ is a proper subset of~$W$ with~$V$ simply connected and~$W$ convex.
Hence, by Lemma~\ref{lemma-hull}, $\lambda(V)>\lambda(W)$, in contradiction with \eqref{E:2.33}.
\end{proof}


\begin{proofsect}{Proof of Theorem~\ref{thm-1} (completed)}
In order to get \eqref{E:1.15}, we note that for any $U\in\CU$,
\begin{equation}
\label{E:3.41ue}
\distH(U,\CM)>\epsilon\quad\Rightarrow\quad\distH(U,\wh U)>\epsilon/2\quad\text{or}\quad\distH(\wh U,\CM)>\epsilon/2.
\end{equation}
If no~$\delta>0$ in the statement of the theorem existed, we would have a sequence $\{V_n\}\subset\CU$ such that $\distH(V_n,\CM)>\epsilon$ yet $\CF(V_n)\to\min\CF$. Proposition~\ref{prop-2.10} then shows that, for~$\wh V_n$ being the convex hull of~$V_n$, we have $\dist(V_n,\wh V_n)\to0$ and so by \eqref{E:3.41ue} above, $\distH(\wh V_n,\CM)>\epsilon/2$ must hold for all but a finite number of~$n$. Reducing to a subsequence permits us to assume that~$\wh V_n\to W$ in the Hausdorff distance and the same argument in the proof of Proposition~\ref{prop-2.10} then shows~$\CF(W)=\min\CF$ and yet $\distH(W,\CM)\ge\epsilon/2$; in contradiction with Theorem~\ref{thm-0}.
\end{proofsect}

\section{Singular features: Facets and corners}
\label{sec4}\nopagebreak\noindent
Here we will address the appearance of various singular features on the minimizing shape. This section contains the proofs of Theorems~\ref{prop-2.12},~\ref{prop-2.16} and~\ref{prop-1.2} (in this order).

\subsection{Facets}
We begin with the appearance of facets. The following lemma will be quite useful:

\begin{lemma}
\label{lemma-2.13}
Let $e$ and~$e'$ be orthogonal unit vectors and let $\theta^\pm=\theta^\pm(e,e')$ be as defined in \eqref{E:degen}. Then for any $t_1,\dots,t_n\in\BR$ and $u_1,\dots,u_n\in\BR$ such that $\sum_{i=1}^n u_i=0$,
\begin{equation}
\label{E:2.40a}
\sum_{i=1}^n\rho(t_ie+u_i e')\ge\Bigl(\,\sum_{i=1}^n t_i\,\Bigr)\rho(e)+\frac{\theta^+-\theta^-}2\sum_{i=1}^n|u_i|.
\end{equation}
Moreover, if also $t_i>0$ for all~$i=1,\dots,n$, then in fact
\begin{equation}
\label{E:2.40b}
\frac\textd{\textd s^+}\sum_{i=1}^n\rho(t_ie+s\,u_i e')\Bigl|_{s=0}=\frac{\theta^+-\theta^-}2\sum_{i=1}^n|u_i|.
\end{equation}
\end{lemma}

\begin{proofsect}{Proof}
Let $t,u\in\BR$ and let $e$ and~$e'$ be orthogonal unit vectors. Let $\theta^\pm=\theta^\pm(e,e')$ be as defined above. We claim that
\begin{equation}
\label{E:2.41}
\rho(te+ue')\ge t\rho(e)+\theta^{\text{sgn}(u)}u.
\end{equation}
(Here and henceforth, $\text{sgn}(u)$ should be interpreted as the symbol $\pm$ depending on the sign of~$u$; in all expressions the sign can be arbitrary when $u$ vanishes.)
Indeed, if~$t>0$ we have
\begin{equation}
\rho(te+ue')- t\rho(e)=t\bigl[\rho(e+ut^{-1}e')-\rho(e)\bigr]\ge t\theta^{\text{sgn}(u)} ut^{-1}=\theta^{\text{sgn}(u)}u,
\end{equation}
where in the inequality we used the convexity of~$s\mapsto\rho(e+se')$. If~$t\le0$ we in turn write
\begin{equation}
\rho(te+ue')- t\rho(e)=\rho(te+ue')+\rho(-te)\ge\rho(ue')=|u|\rho(e'),
\end{equation}
where the middle inequality follows from the triangle inequality and the homogeneity of the norm. To get \eqref{E:2.41}, we observe that
\begin{equation}
\rho(e')\ge\theta^{\text{sgn}(u)}\text{sgn}(u)
\end{equation}
due to the inequalities $\rho(e')-\theta^+\ge0$ and $\rho(e')+\theta^-\ge0$, implied by \eqref{E:2.37}.

Rewriting \eqref{E:2.41} as
\begin{equation}
\rho(te+ue')\ge t\rho(e)+\frac{\theta^++\theta^-}2\,u+\frac{\theta^+-\theta^-}2\,|u|
\end{equation}
we now plug $(t_i,u_i)$ for $(t,u)$, sum over~$i=1,\dots,n$ and invoke the condition $\sum_{i=1}^u u_i=0$ to get \eqref{E:2.40a}. To get \eqref{E:2.40b}, we observe that, for~$t>0$, convexity implies
\begin{equation}
\frac{\textd}{\textd s^+}\rho\bigl(te+sue'\bigr)\Bigl|_{s=0}=\theta^{\text{sgn}(u)}u.
\end{equation}
From here \eqref{E:2.40b} follows by the same argument we just used to prove \eqref{E:2.40a}.
\end{proofsect}

The analysis of the boundary issues will also require some well-known facts concerning the regularity of the principal eigenfunction:

\begin{lemma}
\label{lemma-2.14}
Let $U\in\CU$ and let $\gamma$ be the curve parametrizing~$\partial U$. Let~$h$ be the principal eigenfunction of the Laplacian in~$U$. Then the following holds:
\begin{enumerate}
\item[(1)] $h\in C^\infty(U)$
\item[(2)] if $\gamma$ is $C^{1,\alpha}$ for some $\alpha>0$ in the interval $(a,b)$ then $h$ is $C^{1,\alpha}$ on $U\cup\gamma((a,b))$.
\end{enumerate}
In addition, for $U$ convex we have:
\begin{enumerate}
\item[(3)] For each $0<r<R<\infty$ there is a constant $C=C(R,r)$ such that if~$U$ contains a disc of radius~$r$ and is contained in a disc of radius~$R$, then $\sup_{x\in U}|\nabla h(x)|\le C$.
\end{enumerate}
\end{lemma}

\begin{proofsect}{Proof (sketch)}
(1) The $C^\infty$-property of the eigenfunctions is a consequence of elliptic regularity theory; see Evans~\cite[Section~6.3]{Evans}. 

(2) This is a restatement of Corollary 8.36 of Gilbarg and Trudinger~\cite{GT1998}. 

(3) The regularity of the eigenfunction is usually tied to smoothness of~$\partial U_0$ (see, e.g., Evans~\cite[Section~6.3.2]{Evans} or~\cite{GT1998} above). Other results that work solely under convexity only address part of the boundary (Grieser and Jerison~\cite[Lemma 3.2]{grieser}). We will give a proof based on a regularity estimate for the Green function.

Let $G^U$ denote the Green function in~$U$. Then
\begin{equation}
\label{E:5.2}
h(x)=\lambda(U)\int_U G^U(x,y)h(y)\textd y,\qquad x\in U.
\end{equation} 
By Proposition~1 of Fromm~\cite{Fromm} which is itself based on Theorems~3.3 and~3.4 of Gr\"uter and Widman~\cite{GW}, there is a constant $c(R)$ depending only on the diameter of~$U$ such that
\begin{equation}
\label{E:4.10ue}
\bigl|\nabla_1G^U(x,y)\bigr|\le \frac{c(R)}{|y-x|},\qquad x,y\in U,\, x\ne y,
\end{equation}
where $\nabla_1$ denotes the gradient in the first coordinate. Next we derive a uniform bound $h(x)$. Indeed, using the Cauchy-Schwarz estimate in \eqref{E:5.2} and appling that~$\Vert h\Vert_2=1$ we get
\begin{equation}
\bigl|h(x)\bigr|\le\lambda(U)\int_U G^U(x,y)^2\textd y.
\end{equation}
Invoking the standard estimates $|G^U(x,y)|\le c'\log R$ and $\lambda(U)\le c''/r^2$, where $c',c''\in(0,\infty)$ are numerical constants, and the boundedness of~$U$ we get
\begin{equation}
\label{E:4.11ue}
\Vert h\Vert_\infty\le c \frac{(R\log R)^2}{r^2}
\end{equation}
for some absolute constant~$c\in(0,\infty)$.
Taking $\nabla_1$ on both sides of \eqref{E:5.2} and invoking \eqref{E:4.10ue}, \eqref{E:4.11ue} and the boundedness of~$U$, we get the desired uniform boundedness of $\nabla h$.
\end{proofsect}

Another fact we will need is Hadamard's variational formula for the principal eigenvalue~$\lambda(U)$. We will follow the setting of Henrot~\cite{henrot}. Let $v\colon\BR^2\to\BR^2$ be a bounded and continuous vector field and for $t\in\BR$ define $\Phi_t(x):=x+tv(x)$. Given a bounded open set $U\subset\BR^2$, let $U_t:=\Phi_t(U)$. The set $U_t$ is bounded for all~$t$; we are also going to assume that~$U_t$ is open for~$t$ small. We are interested in the dependence of $\lambda(U_t)$ on~$t$.

We will only need to examine the case when~$U$ is a convex domain. In this case, there is a well defined unit outer normal $n(x)$ at every~$x$ (just take the vector orthogonal to the right-derivative of the boundary parametrized by Euclidean arc-length). We will also need the notation $H^1_{\partial U}$ for the one-dimensional Hausdorff measure  on~$\partial U$ --- namely, the measure that assigns a connected segment $A\subset\partial U$ its Euclidean arc-length. We then have:

\begin{lemma}[Hadamard's variational formula]
Suppose~$U\in\CU$ is convex and let~$U_t$ be defined from a bounded continuous vector field $v\colon\BR^2\to\BR^2$ as above. Then
\begin{equation}
\label{eq:hadamard}
\frac{\textd}{\textd t}\lambda(U_t)\Bigl|_{t=0}=-\int_{\partial U}H^1_{\partial U}(\textd x)\,\bigl|\nabla h(x)\bigr|^2\, v(x)\cdot n(x),
\end{equation}
where $n(x)$ is the outer normal to~$\partial U$ at~$x$ and the integrand is bounded by our assumptions on~$v(x)$ and Lemma~\ref{lemma-2.14}(3).
\end{lemma}

\begin{proofsect}{Proof}
See, e.g., Henrot~\cite[Theorem~2.5.1]{henrot}.
\end{proofsect}

These tools permit us to address the connection between the facets and degenerate directions of the norm:

\begin{proofsect}{Proof of Theorem~\ref{prop-2.12} (facets imply degeneracy)}
Let~$U_0$ be a minimizer of~$\CF$ for a given norm~$\rho$ and let~$\gamma$ be the curve parametrizing~$\partial U_0$. Suppose that~$\partial U_0$ contains a linear segment in direction of a unit vector~$e$. The segment corresponds to the image (under~$\gamma$) of interval~$[a,b]$ of~$t$; we assume that the parametrization is such that~$t\mapsto\gamma(t)$ is linear for~$t\in[a,b]$ and $\gamma(b)-\gamma(a)$ is the Euclidean length of the segment. 

Let~$e'$ be a unit vector orthogonal to~$e$ which we will assume to be oriented as an outward normal to~$\partial U_0$ along $\gamma([a,b])$. Pick a convex~$C^1$-function~$\psi\colon[a,b]\to[0,\infty)$ with $\psi(a)=0=\psi(b)$ and define
\begin{equation}
\gamma_s(t):=\begin{cases}
\gamma(t)+s\psi(t)e',\qquad&\text{if }t\in[a,b],
\\
\gamma(t),\qquad&\text{else}.
\end{cases}
\end{equation}
The assumptions ensure that, for~$s\ge0$ small, $\gamma_s$ parametrizes  the boundary of a set~$U_s\in\CU$. Obviously, $\distH(U_s,U_0)\to0$ as~$s\downarrow0$.

We will now examine the $s$-dependence of~$\lambda(U_s)$ and~$\CP(U_s)$ for~$s\ge0$. The Hadamard formula \eqref{eq:hadamard} with $v(x)=\psi(t)e'$ whenever $x=\gamma(t)$ for $t\in[a,b]$ gives
\begin{equation}
\frac\textd{\textd s^+}\lambda(U_s)\Bigl|_{s=0}=-\int_a^b\bigl|\nabla h(\gamma(t))\bigr|^2\psi(t)\,\textd t,
\end{equation}
where~$h$ is the normalized eigenfunction of the Laplacian in~$U$.
On the other hand, since $\psi$ is smooth, the second part of Lemma~\ref{lemma-2.13} shows
\begin{equation}
\frac\textd{\textd s^+}\CP(U_s)\Bigl|_{s=0}=\,\frac{\theta^+-\theta^-}2\int_a^b\bigl|\psi^{\,\prime}(t)\bigr|\,\textd t
\end{equation}
with $\theta^\pm:=\theta^\pm(e,e')$.
Since the linear segment has the so called uniform ball property everywhere in its interior, the Hopf lemma (see, e.g., the first Lemma in Section~6.4.2 of Evans~\cite{Evans}) yields
\begin{equation}
\bigl|\nabla h(\gamma(t))\bigr|>0,\qquad t\in(a,b).
\end{equation}
Now we just note that, if $\theta^+=\theta^-$, then for $\psi$ non-negative but non-zero
\begin{equation}
\frac{\textd}{\textd s^+}\CF(U_s)\Bigl|_{s=0}<0
\end{equation}
and thus $\CF(U_s)<\CF(U_0)$ for~$s>0$ small. But this is impossible due to the minimality of~$U_0$ and so we must have $\theta^+>\theta^-$. We conclude that, indeed, a linear segment on~$\partial U$ in direction~$e$ implies that~$e$ is degenerate for~$\rho$.
\end{proofsect}

The argument for the converse is relatively similar:

\begin{proofsect}{Proof of Theorem~\ref{prop-2.12} (degeneracy implies facets)}
Let~$U_0$ denote a minimizer of~$U\mapsto\CF(U)$ for a norm~$\rho$. Suppose that~$U_0$ does \emph{not} have a facet in direction~$e$. Then one can find a line in direction~$e$ that intersects the closure of~$U_0$ at exactly one (boundary) point~$x_0$. Let $e'$ be a unit vector orthogonal to~$e$ oriented in the outward normal direction to~$\partial U_0$ at~$x_0$. Define
\begin{equation}
V_s:=\bigl\{x\in U_0\colon (x-x_0)\cdot e'<-s\bigr\},\qquad s\ge0.
\end{equation}
For~$s$ small, $V_s$ is a non-empty convex subset of~$U_0$ with $\distH(V_s,U_0)\to0$ as $s\downarrow 0$. Let~$T_s$ denote the linear segment given by the intersection of the line $(x-x_0)\cdot e'=-s$ with~$U_0$ and let $H^1_{T_s}$ denote the $1$-dimensional Hausdorff measure on~$T_s$. As before, the argument will need a careful examination of the dependence of~$\lambda(V_s)$ and~$\CP(V_s)$ on~$s$.

For the perimeter, \eqref{E:2.40a} shows
\begin{equation}
\label{E:2.54}
\CP(U_0)\ge\CP(V_s)+\frac{\theta^+-\theta^-}2\,s,
\end{equation}
where~$\theta^\pm:=\theta^\pm(e,e')$ are as defined above. On the other hand by integrating \eqref{eq:hadamard} we get 
\begin{equation}
\label{E:2.55}
\lambda(V_s)=\lambda(U_0)+\int_0^s\textd u\int_{T_u}H^1_{T_u}(\textd x)\,\bigl|\nabla h_u(x)\bigr|^2
\end{equation}
where~$h_s$ denotes the principal eigenfunction of the Laplacian in~$U_s$. But Lemma~\ref{lemma-2.14}(3) ensures that $\nabla h_s$ is bounded uniformly in~$s$ sufficiently small and so
\begin{equation}
\int_0^s\textd u\int_{T_u}H^1_{T_u}(\textd x)\,\bigl|\nabla h_u(x)\bigr|^2
\le C|U_0\smallsetminus V_s|
\end{equation}
for some constant~$C\in(0,\infty)$.

The assumption that~$U_0$ has no facet in direction~$e$ means that $|T_s|\to0$ as $s\downarrow0$. Considering the rectangle $R_s$ containing $U_0\smallsetminus V_s$ with one side at~$T_s$ and the other side of length~$s$, the inequality $|U_0\smallsetminus V_s|\le |R_s|=s|T_s|$ yields
\begin{equation}
\label{eq:voldiff}
\lim_{s\downarrow0}\frac{|U_0\smallsetminus V_s|}s = 0.
\end{equation}
This implies $\lambda(V_s)\le\lambda(U_0)+o(s)$ as~$s\downarrow0$ and so, in light of \eqref{E:2.54}, $\CF(V_s)\ge\CF(U_0)$ forces $\theta^+=\theta^-$, i.e., $e$ is \emph{not} a degenerate direction.
\end{proofsect}

\subsection{Corners}
We will now also address the appearance of corners.

\begin{proofsect}{Proof of Theorem~\ref{prop-2.16} (corners imply non-convexity)}
Suppose~$U_0$ has a corner at~$x_0$ and let $v^\pm:=v^\pm(x_0)$. If~$\gamma$ is the parametrization of~$\partial U_0$ using the Euclidean arc-length and $t_0$ is such that $\gamma(t_0)=x_0$, then we have
\begin{equation}
\label{E:2.62}
\gamma(t)=\begin{cases}
x_0+(t-t_0)\nu^-+o(t-t_0),\qquad&\text{if }t<t_0,
\\
x_0+(t-t_0)\nu^++o(t-t_0),\qquad&\text{if }t>t_0.
\end{cases}
\end{equation}
Pick $a,b>0$ and $s>0$ and consider the curve $\xi_s$ defined by
\begin{equation}
\label{E:4.21u}
\xi_s(t):=
\begin{cases}
\frac{t_0-t+bs}{s(a+b)}\gamma(t_0-as)+\frac{t-t_0+as}{s(a+b)}\gamma(t_0+bs),\qquad&\text{if }-as\le t-t_0\le bs,
\\
\gamma(t),\qquad&\text{else}.
\end{cases}
\end{equation}
For~$s$ small non-negative, $\xi_s$ is the boundary of a convex set~$V_s\subset U_0$. We will again need to quantify the $s$-dependence of the perimeter and eigenvalue of~$V_s$.

Let $T_s$ denote the part of the boundary of~$V_s$ corresponding to the first line in \eqref{E:4.21u}. For the eigenvalue we again get by \eqref{eq:hadamard},
\begin{equation}
\lambda(V_s)=\lambda(U_0)+\int_0^s\textd u\int_{T_u}H^1_{T_u}(\textd x)\,\bigl|\nabla h_u(x)\bigr|^2
\end{equation}
and so, by Lemma~\ref{lemma-2.14}(3), $\lambda(V_s)\le\lambda(U_0)+C|U_0\smallsetminus V_s|$. Since $U_0$ has a (convex) corner at~$x_0$, we have $|U_0\smallsetminus V_s|=O(s^2)$ and so
\begin{equation}
\lambda(V_s)\le\lambda(U_0)+O(s^2).
\end{equation}
For the perimeter, \eqref{E:2.62} and continuity of the norms show
\begin{equation}
\CP(V_s)=\CP(U_0)+s\Bigl(\rho(a v^++bv^-)-a\rho(v^+)-b\rho(v^-)+o(1)\Bigr).
\end{equation}
Taking~$s\downarrow0$, the fact that~$U_0$ is the minimizer thus implies
\begin{equation}
\label{E:4.26uu}
\rho(a v^++bv^-)\ge a\rho(v^+)+b\rho(v^-).
\end{equation}
Equality then holds thanks to the triangle inequality.
\end{proofsect}

\begin{proofsect}{Proof of Theorem~\ref{prop-2.16} (non-convexity implies corners)} 
Assume now that $\nu^+$ and~$\nu^-$ are two distinct unit vectors in~$\BR^2$ for which equality holds in \eqref{E:4.26uu} for all $a,b\ge0$. As can be checked, these vectors cannot be collinear. Reversing the orientations if necessary, we may assume that $\nu^+$ is rotated counterclockwise relative to~$\nu^-$ by angle less than~$\pi$.

Let now~$U_0$ denote a minimizing shape of~$\CF$ and consider tangent lines to~$\partial U_0$ in directions $\nu^-$ and~$\nu^+$. These lines touch~$\partial U_0$ at points~$x_1$ and~$x_2$, respectively, and they intersect at a point~$x_0$ (see Figure \ref{fig2}). Now consider the domain~$U$ which is obtained by replacing the portion of~$\partial U_0$ between~$x_1$ and~$x_2$ (if run around counterclockwise) by the line segments $[x_1,x_0]$ and~$[x_0,x_2]$.

\begin{figure}[t]
\refstepcounter{obrazek}
\label{fig2}
\vspace{.2in}
\includegraphics[width=4in]{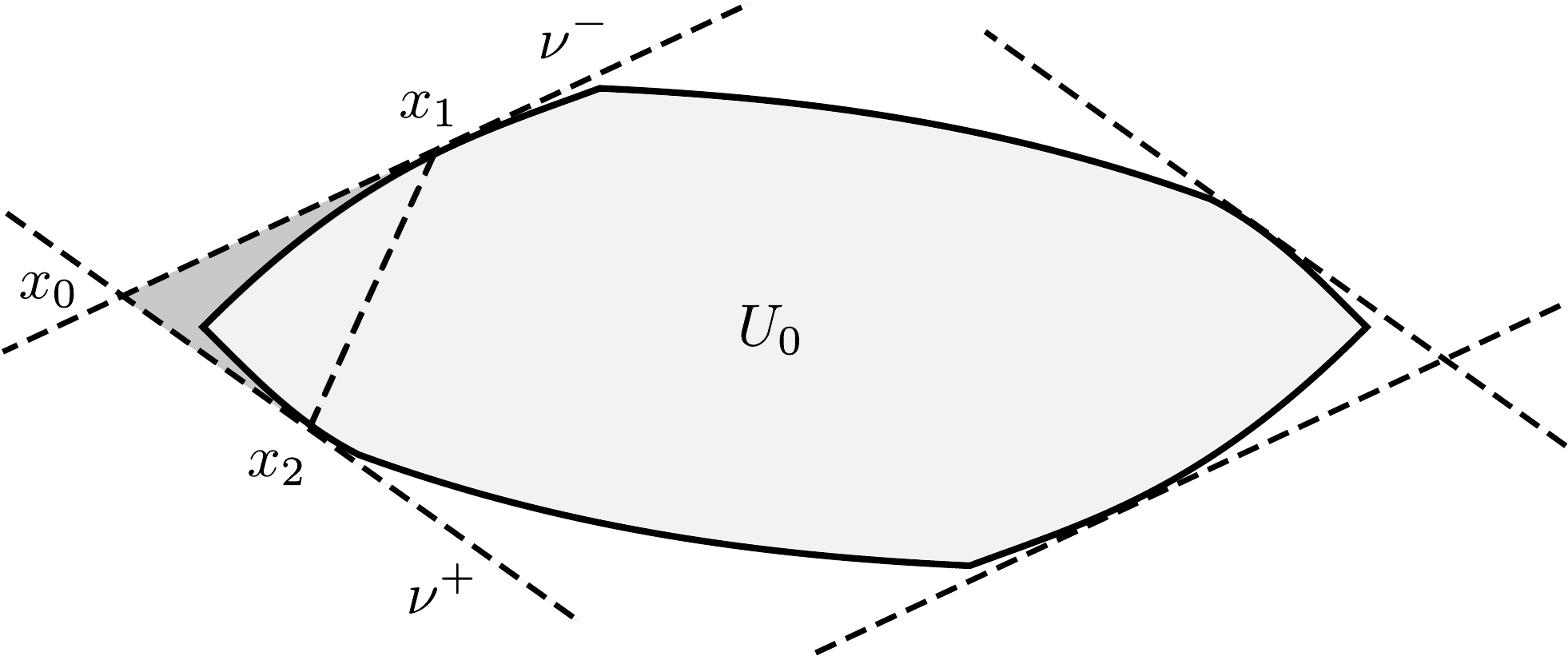}
\vspace{.2in}
\bigskip
\begin{quote}
\fontsize{9}{5}\selectfont
{\sc Fig.~\theobrazek.\ } The construction of $U$ in the second part of the proof of Theorem~\ref{prop-2.16}. The underlying set is necessarily symmetric with respect to $x\mapsto -x$, so if a corner can be added on one side of the set, a symmetric corner would appear on the other side as well.
\normalsize
\end{quote}
\end{figure}
The construction ensures that~$U$ is a bounded, open, convex set with a corner at~$x_0$ where the tangent vectors are equal~$\nu^-$ and~$\nu^+$, respectively. If~$x_0\in\partial U_0$, then $U=U_0$ and we would be done so let us assume that $x_0\not\in\partial U_0$ and derive a contradiction.

Clearly, $U\supseteq U_0$ and since~$U$ is convex while $U_0\in\CU$ and $U\ne U_0$, we have~$\lambda(U)<\lambda(U_0)$ by Lemma~\ref{lemma-hull}. The triangle inequality implies that the $\rho$-length of the portion of~$\partial U$ between~$x_1$ and~$x_2$ is at least~$\rho(x_2-x_1)$. But
\begin{equation}
x_2-x_1 = x_2-x_0+(x_0-x_1)
\end{equation}
and so, by the equality in \eqref{E:4.26uu},
\begin{equation}
\rho(x_2-x_1) = \rho(x_2-x_0)+\rho(x_0-x_1).
\end{equation}
It follows that~$\CP(U)\le\CP(U_0)$ and thus $\CF(U)<\CF(U_0)$, in contradiction with the minimality of~$U_0$. Hence,~$U_0$ has a corner at~$x_0$ after all.
\end{proofsect}

\subsection{Non isoperimetric minimizers}
\label{sec-last}\noindent
We end this section by analyzing the situations for which the minimizer of $\CF$ is not homothetic to the   minimizer of the corresponding isoperimetric problem~\eqref{E:1.7w}.

\begin{proofsect}{Proof of Theorem~\ref{prop-1.2} (existence of a counterexample)}
Not too surprisingly, a counterexample is provided by an asymmetric $\ell^1$-norm. Fix $n\ge1$ and, for $x=(x_1,x_2)$ let
\begin{equation}
\rho(x):=\frac{1}{n}|x_1|+n|x_2|.
\end{equation}
Among sets of unit Lebesgue area, the perimeter is minimized by the rectangle $[0,n]\times[0,1/n]$. Next pick $a>0$ and let denote $U_n^a:=[0,n/a]\times[0,a/n]$. Using separation of variables (see, e.g., Henrot~\cite[Proposition 1.2.13]{henrot}) one finds $\lambda(U_n^a)=\pi^2\left(\frac{a^2}{n^2}+\frac{n^2}{a^2}\right)$. Thus
$$
\CF(tU_n^a)=t^{-2}\pi^2\left(\frac{a^2}{n^2}+\frac{n^2}{a^2}\right)+2t\left(\frac{1}{a}+a\right)
.$$
Abbreviating $A:= \pi^2(\frac{a^2}{n^2}+\frac{n^2}{a^2})$ and $B:=\frac{1}{a}+a$, we then get
$$
\min_{t>0}\CF(tU_n^a)=3A^{1/3}B^{2/3}
.$$
Since $AB^2\sim\pi^2n^2(1+a^{-2})^2$ in the limit~$n\to\infty$, taking $a>1$ will result in a smaller value of $\min_t\CF(U_n^a)$ than~$a:=1$, provided~$n$ is large enough. But for~$a=1$ the isoperimetric minimizer~$U_n^1$ is the minimizer of~$U\mapsto\CF(U)$, while $U_n^a$ is homothetic to~$U_n^{a'}$ if and only if~$a=a'$. Hence, the isoperimetric minimizer is not homothetic to the minimizer of $\CF$ once~$n$ is large enough.
\end{proofsect}

\begin{proofsect}{Proof of Theorem~\ref{prop-1.2} (smooth shapes)}
Let $h$ denote the $L^2$-normalized eigenfunction corresponding to the principal eigenvalue of the (negative) Laplacian in~$U_0$ with Dirichlet boundary conditions on~$\partial U_0$.
The assumptions on~$\rho$ and~$U_0$ permit us to take variational derivative (see \cite[Theorem 2.2]{BBH2011} and \cite{Giusti,Variation}) and find out that we must have
\begin{equation}
\label{E:1.20c}
\bigl|\nabla h(x)\bigr|^2=C_\rho(x),\qquad x\in\partial U_0,
\end{equation}
where~$C_\rho(x)$ is the curvature of~$\partial U_0$ at point~$x$ relative to arc-length defined by the norm~$\rho$ and where $\nabla h$ extends continuously to~$\partial U_0$ by Lemma~\ref{lemma-2.14}(2). The corresponding variational derivative of the Lebesgue area of~$U_0$ is a constant and so if~$U_0$ is, even just up to a scaling, also a minimizer of the isoperimetric problem, a similar reasoning yields
\begin{equation}
C_\rho(x)=\text{constant},\qquad x\in\partial U_0.
\end{equation}
But this means that $|\nabla h|$ is constant everywhere on~$\partial U_0$. Since~$U_0$ has $C^1$-boundary the gradient~$\nabla h(x)$ is collinear with the (everywhere-defined) outer normal vector $\hat n = \hat n(x)$. It follows that~$h$ satisfies the overdetermined elliptic problem
\begin{equation}
\begin{alignedat}{3}
-\Delta h &=\lambda(U_0)h,\qquad&&\text{in }U_0,
\\
h&=0,\qquad&&\text{on }\partial U_0,\\
\nabla h&=c \,\hat n,\qquad&&\text{on }\partial U_0,
\end{alignedat}
\end{equation}
for some constant~$c<0$. A classic result due to Serrin~\cite{Serrin1971} states that (for~$C^2$-domains) this is only possible if~$U_0$ is a Euclidean ball. By the duality characterization of the isoperimetric minimizers, $\rho$ must be a multiple of the Euclidean norm.
\end{proofsect}

\section*{Acknowledgments}
\noindent
This research has been partially supported by NSF grant DMS-1407558 and GA\v CR project P201/16-15238S. The authors wish to thank John Garnett, Inwon Kim and Peter Kuchment for discussions at various stages of this project.

\bibliographystyle{abbrv}

\end{document}